\documentclass[12pt]{article}

\usepackage{amsfonts,amssymb,amsmath,amsthm,epsfig,euscript}

\newtheorem{theorem}{Theorem}
\newtheorem{lemma}[theorem]{Lemma}

\newtheorem{remark}{Remark}
\newtheorem{proposition}{Proposition}

\theoremstyle{definition}
{

}

\long\def\symbolfootnote[#1]#2{\begingroup
\def\thefootnote{\fnsymbol{footnote}}\footnote[#1]{#2}\endgroup}

\newcommand{\sg}{\sigma}

\newcommand{\bx}[1][\sigma]{\mathrm{1\mbox{-}box}(#1)}

\newcommand{\bond}[1][\sigma]{\mathrm{bond}(#1)}

\newcommand{\fig}[2]{\begin{figure}[ht]
\centerline{\scalebox{.66}{\epsfig{file=#1.eps}}}
\caption{#2}
\label{fig:#1}
\end{figure}}

\title{$(a,b)$-rectangle patterns in permutations and words}

\author{
Sergey Kitaev \\
\small Department of Computer and Information Sciences\\[-0.8ex]
\small University of Strathclyde\\[-0.8ex]
\small Livingstone Tower, 26 Richmond Street\\[-0.8ex]
\small Glasgow G1 1XH, United Kingdom\\[-0.8ex]
\small \texttt{sergey.kitaev@cis.strath.ac.uk}
\and
Jeffrey Remmel \\
\small Department of Mathematics\\[-0.8ex]
\small University of California, San Diego\\[-0.8ex]
\small La Jolla, CA 92093-0112. USA\\[-0.8ex]
\small \texttt{jremmel@ucsd.edu}
}

\date{\small Submitted: Date 1;  Accepted: Date 2;
 Published: Date 3.\\
\small MR Subject Classifications: 05A15, 05E05}

\begin{document}
\maketitle

\begin{abstract}
\noindent \ In this paper, we introduce the notion of a $(a,b)$-rectangle pattern on permutations that not only generalizes the notion of successive elements (bonds) in permutations, but is also related to mesh patterns introduced recently by Br\"and\'en and Claesson. We call the $(k,k)$-rectangle pattern the 
$k$-box pattern. To provide an enumeration result on the maximum number of occurrences of the $1$-box pattern, we establish an enumerative result on pattern-avoiding signed permutations. 

Further, we extend the notion of $(k,\ell)$-rectangle patterns to words and binary matrices, and provide distribution of $(1,\ell)$-rectangle patterns on words; explicit formulas are given for up to 7 letter alphabets where $\ell \in \{1,2\}$, while obtaining distributions for larger alphabets depends on inverting a matrix we provide. We also provide similar results for the distribution of bonds over words. As a corollary to our studies we confirm a conjecture of Mathar on the number of ``stable LEGO walls'' of width 7 as well as prove three conjectures due to Hardin and a conjecture due to Barker. We also enumerate two sequences published by Hardin in the On-Line Encyclopedia of Integer Sequences.\\

\end{abstract}

\section{Introduction}

The notion of mesh patterns was introduced by Br\"and\'en and Claesson \cite{BrCl} to provide explicit expansions for certain permutation statistics as, possibly infinite, linear combinations of (classical) permutation patterns (see \cite{kit} for a comprehensive introduction to the theory of patterns in permutations and words).  This notion was studied further in a series of papers, e.g. in  \cite{AKV,kitlie,kitrem,kitrem2,Ulf}. 

In this paper, we introduce the notion of an {\em 
$(a,b)$-rectangle} patterns in permutations, words and binary matrices. That is, let $\sigma = \sg_1\sg_2 \ldots \sg_n\in S_n$ be a permutation written in one-line notation, where $S_n$ denotes the set of all permutations of length $n$. Then we will consider the 
graph of $\sg$, $G(\sg)$, to be the set of points $(i,\sg_i)$ for 
$i =1, 2,\ldots, n$.  For example, the graph of the permutation 
$\sg = 471569283$ is pictured in Figure 
\ref{fig:basic}.  

\fig{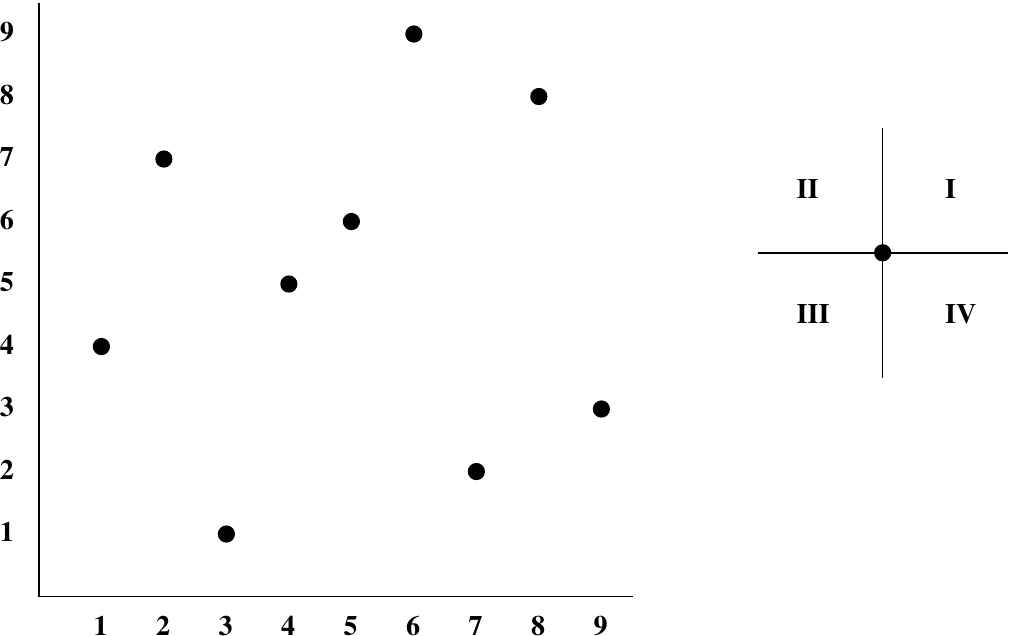}{The graph of $\sg = 471569283$.}

Then if we draw a coordinate system centered at a 
point $(i,\sg_i)$, we will be interested in  the points that 
lie in the $(2a) \times (2b)$ rectangle centered at the origin, that is, 
in the set of points  
$(i\pm r,\sg_i \pm s)$ such that $r\in \{0, 1,\ldots, a\}$ and 
$s \in \{0, 1,\ldots, b\}$.  We say that $\sg_i$ {\em matches the $(a,b)$-rectangle  pattern} in $\sg$, if there is 
at least one point in the $(2a) \times (2b)$ rectangle centered at the point 
$(i,\sg_i)$ in $G(\sg)$ other than $(i,\sg_i)$. For example, when 
we look for matches of the (2,3)-rectangle patterns, we would 
look at $4 \times 6$ rectangles centered at points $(i,\sg_i)$ as 
pictured in Figure \ref{fig:basic4} for a particular point.

\fig{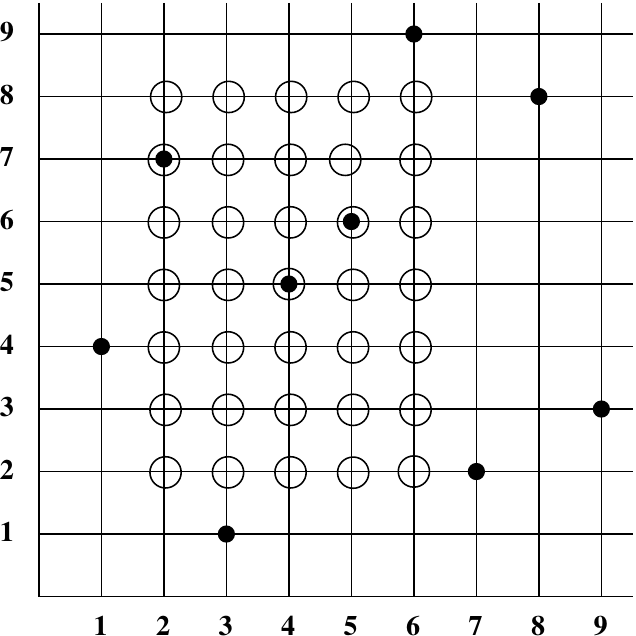}{The $4 \times 6$ rectangle centered at the point 
$(4,5)$ in the graph of $\sg = 471569283$.}

We shall refer to the $(k,k)$-rectangle pattern as 
the $k$-box pattern. 
For example, if $\sg = 471569283$, then the 2-box centered 
at the point $(4,5)$ in $G(\sg)$ is the set of circled 
points pictured in Figure 
\ref{fig:basic3}.  Hence, 
$\sg_i$ matches the $k$-box pattern in $\sg$, if there is 
at least one point in the $k$-box centered at the point 
$(i,\sg_i)$ in $G(\sg)$ other than $(i,\sg_i)$. 
For example, $\sg_4$ matches 
the pattern $k$-box for all $k \geq 1$ in $\sg = 471569283$ since 
the point $(5,6)$ is present in the $k$-box centered at the point $(4,5)$ in 
$G(\sg)$ for all $k \geq 1$. 
However, $\sg_3$ only matches the $k$-box pattern 
in  $\sg = 471569283$ for $k \geq 3$ since there are no points 
in 1-box or 2-box centered at $(3,1)$ in $G(\sg)$, but 
the point $(1,4)$ is in the 3-box centered at $(3,1)$ in $G(\sg)$. 
For $k \geq 1$, we let $k\mbox{-box}(\sg)$ denote the set of 
all $i$ such that $\sg_i$ matches the $k$-box pattern in 
$\sg = \sg_1\ldots \sg_n$.

\fig{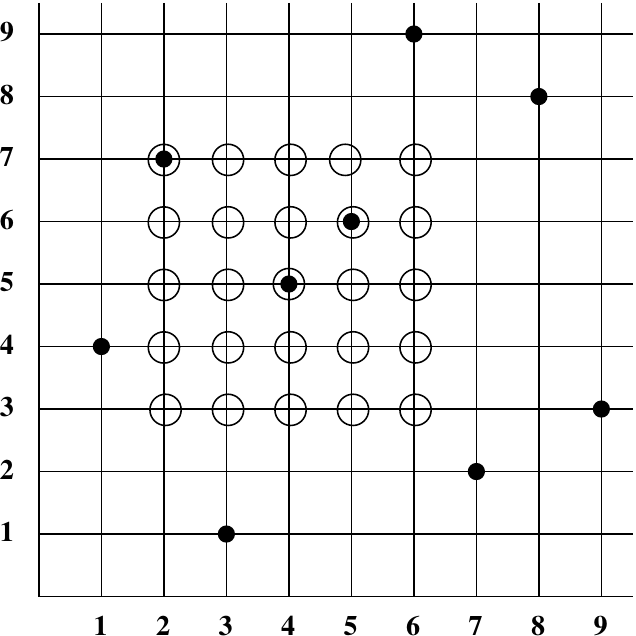}{The 2-box centered at the point 
$(4,5)$ in the graph of $\sg = 471569283$.}

In this paper, we shall mainly be interested in the 
1-box patterns in permutations and words.  Note that $\sg_i$ matches the 1-box pattern in a permutation $\sg=\sg_1\ldots\sg_n$ if 
either $|\sg_i - \sg_{i+1}|=1$ or $|\sg_{i-1}-\sg_i|=1$, while if $\sg$ were a word, $\sg_i$ matches the 1-box pattern if either $|\sg_i - \sg_{i+1}|\leq 1$ or $|\sg_{i-1}-\sg_i|\leq 1$. 
For any permutation $\sg = \sg_1 \ldots \sg_n \in S_n$, let $\bx[\sg]$ denote 
the number of $i$ such that $\sg_i$ matches the 1-box pattern in 
$\sg$.  More generally, $\sg_i$ matches the  $(a,b)$-rectangle pattern in 
$\sg$ if 
there is a $\sg_j$ such that $0 < |i-j] \leq a$ and $|\sg_i -\sg_j| \leq b$. 
We let $(a,b)\mbox{-rec}(\sg)$ denote the number of $i$ such that $\sg_i$
matches the  $(a,b)$-rectangle pattern in $\sg$.

Avoidance of the 1-box pattern is given by permutations without rising or falling successions which are also called bonds.  That is, 
a {\em bond} in a permutation $\sg = \sg_1 \ldots \sg_n \in S_n$ is 
is a pair $\sg_i \sg_{i+1}$ of the form $s (s+1)$ or $(s+1) s$ for 
some $s$. We let $\bond[\sg]$ denote the number of bonds of 
$\sg$. We note that in general $\bx[\sg] \neq \bond[\sg]$. 
For example, if $\sg = 214365$, then 
$\bx[\sg]=6$ while $\bond[\sg]=3$. However, for any permutation $\sg\in S_n$,  $\bx[\sg] =0$ if and only if $\bond[\sg]=0$.

The  distributions of 
$\bx[\sg]$ and $\bond[\sg]$ for $S_2$, $S_3$, and $S_4$ are given below.

\begin{center}
\begin{tabular}{ccc}

\begin{tabular}{|c|c|c|}
\hline
$\sg$ &  $\bx[\sg]$ & $\bond[\sg]$\\
\hline
12 & 2&1\\
\hline
21 & 2 &1\\
\hline
\end{tabular}

& \ \ \ \ \ &

\begin{tabular}{|c|c|c|}
\hline
$\sg$ &  $\bx[\sg]$ & $\bond[\sg]$ \\
\hline
123 & 3 & 2\\
\hline 
132 & 2 & 1\\
\hline
213 & 2 & 1\\
\hline
231 & 2 & 1\\
\hline
312 & 2 & 1 \\
\hline
321 & 3 & 2\\
\hline
\end{tabular}\\

\end{tabular}
\end{center}

\begin{center}
\begin{tabular}{|c|c|c|c|c|c|c|}
\hline
$\sg$ &  $\bx[\sg]$ & $\bond[\sg]$ & \ & $\sg$ & $\bx[\sg]$ & $\bond[\sg]$\\
\hline
1234 & 4 & 3& \ & 2134 & 4 & 2\\
\hline 
1243 & 4 & 2 & \ & 2143 & 4 & 2\\
\hline
1324 & 2 & 1&  \ & 2314 & 2 & 1\\
\hline
1342 & 2 & 1 & \ & 2341 & 3 & 2\\
\hline
1423 & 2 & 1  & \ & 2413 & 0 & 0\\
\hline
1432 & 3 & 2  & \ & \ 2431 & 2 & 1\\
\hline
3124 & 2 & 1 &\ & 4123 & 3 & 2\\
\hline 
3142 & 0 & 0 &\ & 4132 & 2 & 1\\
\hline
3214 & 3 & 2 & \ & 4213 & 2 & 1\\
\hline
3241 & 2 & 1 & \ & 4231 & 2 & 1\\
\hline
3412 & 4 & 2 & \ & 4312 & 4 & 2\\
\hline
3421 & 4  & 2 & \ &  4321 & 4 & 3\\
\hline
\end{tabular}
\end{center}

Finding the number of permutations $\sg$ of length $n$ with $\bond[\sg] =0$ (equivalently, $\bx[\sg]=0$) is equivalent to solving the {\em problem of Hertzsprung}, which is finding the number of ways to arrange $n$ non-attacking kings on an $n\times n$ board, with one in each row and column.   Riordan \cite{Riordan} first derived a recurrence relation for the number $a_n$ of such permutations in 1965: $a_0 = a_1=1$, $a_2=a_3 = 0$, and for $n\geq 4$,  
$$a_n = (n+1)a_{n-1}-(n-2)a_{n-2}-(n-5)a_{n-3}+(n-3)a_{n-4}.$$
The initial values for $a_n$ are $$1, 1, 0, 0, 2, 14, 90, 646, 5242, 47622, 479306, 5296790, 63779034, \ldots.$$
We refer to the sequence A002464 in the {\em On-Line Encyclopedia of Integer Sequences} ({\em OEIS}) for many references and for other interpretations/properties of this sequence of numbers. In particular, the generating function for these numbers was derived by Flajolet:
$$\sum_{n\geq 0} \frac{n!x^n(1-x)^n}{(1+x)^n}.$$

Riordan \cite{Riordan} obtained a more general result.  That is, let $S_{n, m}$ be the number of permutations in $S_n$ with exactly $m$ bonds, and let $S[n]:=S[n](t) = \sum_{m \geq 0} S_{n, m}t^m$. Then $S[0] = 1$, $S[1] = 1$, $S[2] = 2t$, $S[3] = 4t+2t^2$, and for $n\geq 4$,
\begin{small} 
$$S[n] = (n+1-t)S[n-1] - (1-t)(n-2+3t)S[n-2] - (1-t)^2(n-5+t)S[n-3] + (1-t)^3(n-3)S[n-4].$$
\end{small}
In particular, the coefficient of $t$ in $S[n](t)$  gives the number of permutations of length $n$ with exactly one bond, which, in our terminology, is the number of permutations in $S_n$ with exactly two occurrences of the 1-box pattern. This is the sequence A086852 in the OEIS. Clearly, there are no permutations with exactly one occurrence of the 1-box pattern. 

It is straightforward to see that the number of permutations of length $n+1$ with exactly three occurrences of the 1-box pattern is equal to the number of permutations of length $n$ with exactly two occurrences of the 1-box pattern. Indeed, to have exactly three occurrences of the pattern in a permutation $\pi$ means to have in $\pi$ a factor either of the form $a(a+1)(a+2)$ or of the form $(a+2)(a+1)a$, and no other consecutive successive elements. Removing $(a+1)$ from $\pi$ and decreasing by $1$ all elements that are larger than $(a+1)$, we get a permutation containing exactly two occurrences of the 1-box pattern. This procedure is obviously reversible. Thus,  the coefficient of $t$ in $S[n](t)$  also gives the number of permutations of length $n+1$ with exactly three occurrences of the 1-box pattern. 

Hence, our study of 1-box/$k$-box patterns can not only be seen as an extension of the study of mesh patterns, but also as an extension of the study of consecutive successive elements (bonds) conducted in the literature. We do not define the notation of mesh patterns in this paper; however, the relevance of these patterns to our patterns is that in both cases we look for presence of points in specified regions in graphical representation of permutations. 

In Theorem \ref{permsMain}, we will enumerate permutations having the maximum number of occurrences of the 1-box pattern. To achieve this result, we obtain a result on pattern-avoiding signed permutation (see Theorem \ref{barredEnum}) thus contributing to the theory of permutation patterns (see \cite{kit}). 

In Section \ref{secWords} we not only provide a general solution (in matrix form) for finding the distribution of bonds and 1-box patterns over words
(see Theorems \ref{thm:1bWords} and \ref{thm:1boxgf}) but also apply our studies to settle a conjecture of Mathar on the number of ``stable LEGO walls'' of width 7 (see Subsection \ref{subLEGO}), as well as to settle three conjectures of Hardin (see Subsection \ref{threeHardin}) and a conjecture of Barker (see Subsection \ref{secWords2}). Also, in Subsection  \ref{secWords2}, we enumerate two sequences published by Hardin in the OEIS.

Given a word $w_1 \ldots w_n \in [\ell]^n$, where $[\ell]=\{1,2,\ldots,\ell\}$, we say that 
the pair $w_iw_{i+1}$ is a {\em $k$-bond} if  $|w_i-w_{i+1}| \leq k$. In Subsection \ref{secWords2}, we study the distribution of 2-bonds and (1,2)-rectangle patterns in words.

\section{Permutations with the maximum number of occurrences of the 1-box pattern}

It is straightforward to see that the maximum possible number of occurrences of the 1-box pattern in a permutation of length $n$ is $n$ (e.g. the increasing permutation $12\ldots n$ achieves this maximum). 

In order to enumerate permutations with the maximum number of occurrences of the 1-box pattern, we  need the notion of the {\em hyperoctahedral group} $B_n$ whose elements can be regarded as {\em signed permutations} written as $\alpha=\alpha_1\alpha_2\cdots \alpha_n$ in which each of the letters $1,2,\ldots,n$ appears, possibly barred. For example, $B_2=\{12,\overline{1}2,1\overline{2},\overline{1}\overline{2},21,\overline{2}1,2\overline{1},\overline{2}\overline{1}\}$. Clearly, $|B_n|=2^nn!$.

\begin{theorem}\label{barredEnum} The exponential generating function for $a_n$, the number of elements in $B_n$ avoiding factors of the form $i(i+1)$ and $(\overline{i+1})\overline{i}$ simultaneously, is given by
\begin{equation}\label{egf}
A(t)=\sum_{n\geq 2}\frac{a_{n}t^n}{n!}=1+2\int_{0}^{t}\frac{e^{-z}}{(1-2z)^2}\ dz.
\end{equation} 
The initial values for $a_n$ are
$$1,2,6,34,262,2562,30278,419234,6651846,118950658,2366492038,\ldots.$$\end{theorem}

\begin{proof} Clearly, $a_0=1$ and $a_1=2$ since the empty signed permutation, as well as 1 and $\overline{1}$, avoid the prohibited factors. Our goal is to show that for $n\geq 2$,
\begin{equation}\label{recBarred} a_n=(2n-1)a_{n-1}+2(n-2)a_{n-2}.\end{equation}

In what follows, by {\em doubling} an element $i$ (resp. $\overline{i}$) of $\pi\in B_n$ we mean increasing all the elements of $\pi$, if any,  that are greater than $i$ by 1, and substituting $i$ (resp. $\overline{i}$) by $i(i+1)$ (resp. $(\overline{i+1})\overline{i}$). 

Let $b_n$ be the number of elements in $B_n$ with exactly one occurrence of either the factor $i(i+1)$ or $(\overline{i+1})\overline{i}$. We refer to the elements of such an occurrence as a {\em bad pair}. In particular, doubling an element results in appearance of exactly one new bad pair. 

It is easy to see that 
\begin{equation}\label{bANDa} 
b_n=(n-1)a_{n-1}. 
\end{equation}
Indeed, the only way to create an object counted by $b_n$ is to pick an object counted by $a_{n-1}$ and to double one of its elements (exactly one bad pair is then created); this procedure is obviously reversible, since reversing the doubling procedure will never introduce new bad pairs.

Next, we will show the following relation between $a_n$s and $b_n$s:
\begin{equation}\label{aANDb}
a_n=2b_{n-1}+2na_{n-1}-a_{n-1}.
\end{equation}
Indeed, remove the largest element (either $n$ or $\overline{n}$) in $\pi$ counted by $a_n$ to obtain $\pi'$. Since clearly at most one bad pair can be created, either $\pi'$ is counted by $b_{n-1}$ or by $a_{n-1}$. Thus, to generate all objects counted by $a_n$, we either take an object counted by
\begin{itemize} 
\item $b_{n-1}$ and break the bad pair by inserting either $n$ or $\overline{n}$ between the bad pair elements; there are $2b_{n-1}$ ways to do this (note that there are no problems with $n-1$ or $\overline{n-1}$ be involved in the bad pair, since inserting either $n$ or $\overline{n}$ in this case will still not create a new bad pair), or an object counted by 
\item $a_{n-1}$. There are $n$ possible places we can insert either $n$ or $\overline{n}$ giving us $2na_{n-1}$ possibilities. However, inserting $n$ right after $(n-1)$ or inserting $\overline{n}$ right before $\overline{n-1}$ will give us a bad pair, and thus must not be counted: there are $a_{n-1}$ such objects (for each $\pi'$ counted by $a_{n-1}$ there is a unique bad position and a unique choice of the largest element to be inserted to create an object counted by $b_n$ rather than by $a_n$). This completes the proof of (\ref{aANDb}).  
\end{itemize}

Using (\ref{bANDa}) and (\ref{aANDb}) we obtain (\ref{recBarred}).

Note that second derivative of $A(t)$ is given by 
\begin{eqnarray*}
A^{\prime \prime}(t) &=& \sum_{n \geq 0} a_{n+2} \frac{t^n}{n!} = \sum_{n \geq 0} ((2n+3)a_{n+1} + 2na_n) \frac{t^n}{n!} \\
&=& 2t \sum_{n \geq 1} a_{n+1} \frac{t^{n-1}}{(n-1)!} + 
3 \sum_{n \geq 0} a_{n+1}\frac{t^n}{n!} + 2t \sum_{n \geq 1} a_{n} \frac{t^{n-1}}{(n-1)!} \\
&=& 2tA^{\prime \prime}(t) + 3A^{\prime}(t) +2tA^{\prime}(t).
\end{eqnarray*}
Solving for $A^{\prime \prime}(t)$, we see that 
\begin{equation*}
A^{\prime \prime}(t) = \frac{2t+3}{1-2t} A^{\prime}(t)
\end{equation*}
or, equivalently,
\begin{equation}\label{eq:A1}
\frac{A^{\prime \prime}(t)}{A^{\prime }(t)} = -1 + \frac{4}{1-2t}.
\end{equation}
Integrating both sides of (\ref{eq:A1}) and using the 
fact that $A^{\prime}(0) = 2$, we see 
that 
\begin{equation*}
\ln(A^{\prime}(t)) = -t -2\ln(1-2t) + \ln(2).
\end{equation*}
Thus 
\begin{equation}\label{eq:A2}
A^{\prime}(t) =2 \frac{e^{-t}}{(1-2t)^2}.
\end{equation}
Integrating both sides of (\ref{eq:A2}) and using the 
fact that $A(0) = 1$, we see 
that 
\begin{equation*}\label{eq:A3}
A(t) = 1 + 2 \int_{0}^t \frac{e^{-t}}{(1-2t)^2}\mbox{d}t. 
\end{equation*}

\end{proof}

\begin{remark} Theorem \ref{barredEnum} is a result on pattern avoidance in signed permutations (see \cite[Chapter 9.6]{kit} for relevant results). In fact, avoidance of factors of the form $i(i+1)$ and $(\overline{i+1})\overline{i}$ can be expressed in terms of avoidance of {\em bivincular patterns} (see \cite[Chapter 1.4]{kit} for definition; bars can be incorporated in the definition in an obvious way extending it from $S_n$ to $B_n$), and thus  Theorem \ref{barredEnum}  seems to be the first instance of enumerative results on signed permutations avoiding bivincular patterns.  \end{remark}

\begin{theorem}\label{permsMain} The number of permutations in $S_n$ with the maximum number of occurrences of the $1$-box pattern (which is $n$) is given by 
\begin{equation}\label{sumBin}
\sum_{j=1}^{\lfloor\frac{n}{2}\rfloor}{n-j-1\choose j-1}a_j
\end{equation}
where $a_j$'s are given by  the recurrence (\ref{recBarred}) or by the exponential generating function (\ref{egf}). The initial values for the number of such permutations starting with the case $n=0$ are
$$1, 1, 2, 2, 8, 14, 54,128, 498, 1426, 5736, 18814, 78886, 287296, 1258018,\ldots.$$  
\end{theorem}

\begin{proof} Each permutation $\pi\in S_n$ having the maximum number of occurrences of the 1-box pattern can be uniquely decomposed into maximal factors of consecutive elements of size at least 2, since each element of $\pi$ must be staying next to a consecutive element. For example, the permutation $\pi=543126798$ is decomposed into maximal factors 543, 12, 67 and 98. Let a permutation $\pi'$ be obtained from $\pi$ by substituting the $i$th largest block with $i$ if it is increasing, and with $\overline{i}$ if it is decreasing. We refer to $\pi'$ as the {\em basis permutation} for $\pi$ and, clearly, $\pi'\in B_n$ for some $n$. For $\pi$ as above, $\pi'=\overline{2}13\overline{4}$. Since the decomposition factors are of maximal possible length, basis permutations must avoid factors of the form $i(i+1)$ and $(\overline{i+1})\overline{i}$, and these permutations were counted by us in Theorem \ref{barredEnum}. 

Finally, to create permutations of length $n$ with the maximum number of occurrences of the 1-box pattern, we choose basis permutations of length $j$, $1\leq j\leq \lfloor\frac{n}{2}\rfloor$, and decide on the lengths of the $j$ decomposition factors to be made decreasing or increasing depending on the respective elements to have or not to have bars, respectively. These lengths must be of size at least 2, and  it is a standard combinatorial problem to see that the number of ways to make such a decision is ${n-j-1\choose j-1}$ (indeed, we reserve $2j$ elements to make sure each decomposition factor will contain at least two elements; the remaining $n-2j$ elements can be distributed among $j$ factors in the desired number of ways). Note that all permutations of interest will be generated in a bijective manner, which completes our proof of (\ref{sumBin}).\end{proof}

\section{Distribution of bonds and $1$-box patterns over words}\label{secWords}

Given a word $w =w_1 \ldots w_n$, 
let $|w| =n$ be the length of the $w$ and 
$\bx[w]$ denote the number of occurences of the 1-box pattern in $w$. 
A {\em bond} in $w$ is a pair $w_i w_{i+1}$ of the form 
$s(s+1)$, $(s+1)s$, or $ss$ for some $s$. We let 
$\bond[w]$ denote the number of bonds of $w$.  

In Subsection \ref{bondDist} we study distribution of bonds over words, while in Subsection \ref{1boxDist} we study distribution of 1-box patterns over words. Three relevant conjectures of Hardin are settled in Subsection~\ref{threeHardin}, and a conjecture of Mathar on stable LEGO walls is settled in Subsection \ref{subLEGO}. In Subsection \ref{secWords2}, we consider $(1,k)$-rectangle patterns for $k\geq 2$, which led us to solving a conjecture of Barker and enumerating two sequences of Hardin published in the OEIS.

\subsection{Distribution of bonds over words.}\label{bondDist}

As in the case of permutations, it is realatively straightforward 
to find the generating functions for the number of bonds in 
words over $[\ell]$ for any $\ell \geq 1$. That is, let
$$A_{\ell,1}(x,t)= \sum_{w \in [\ell]^*} x^{\bond[w]}t^{|w|} =\sum_{m,n\geq 0}a_{\ell}(m,n)x^mt^n,$$
where $ [\ell]^*$ is the set of all words over the alphabet $[\ell]$. Thus $a_{\ell,1}(m,n)$ is the number of words $w$ of length $n$ over the alphabet $[\ell]$ such that $\bond[w] =m$. 

The following theorem gives the distribution of bonds over words in matrix form.

\begin{theorem}\label{thm:1bWords} 
The generating function $A_{\ell,1}(x,t)$ is equal to $$1+(\underbrace{1,\ldots,1}_{\ell})\mathbb{A}_{\ell,1}^{-1}(\underbrace{-t,\ldots,-t}_{\ell})^{T}$$ where $\mathbb{A}_{\ell,1}$ is the following $\ell\times\ell$ matrix: 

$$\mathbb{A}_{\ell,1}=\left(\begin{array}{cccccccc} 
   xt-1 & xt         &     t      &       t   & t           & \cdots & t & t \\ 
   xt     & xt-1     &    xt     &       t   & t           & \cdots & t & t \\ 
   t       & xt         &    xt-1 &       xt & t           & \cdots & t & t \\ 
 t       &  t         &    xt &       xt-1 & xt           & \cdots & t & t \\ 
\vdots & \vdots & \vdots & \vdots & \vdots & \ddots & \vdots & \vdots \\ 
    t      &       t    &     t      &    t       &  t & \cdots  &     xt    &  xt-1 \end{array}\right).$$

\end{theorem}
 
\begin{proof}

Let $i[\ell]^*$ denote the set of words over $[\ell]$ that begin with a letter $i$. For $1\leq i\leq \ell$, let 
$$A_{\ell,1}^{(i)}(x,t)= \sum_{w \in i[\ell]^*} x^{\bond[w]}t^{|w|} = 
\sum_{m,n\geq 0}a^{(i)}_{\ell,1}(m,n)x^mt^n.$$ 
Thus $a^{(i)}_{\ell,1}(m,n)$ is the number of words of length $n$ over $[\ell]$ such that $w$ begins with the letter $i$ and 
$\bond[w] =m$. Clearly, 
\begin{equation}\label{sumA}
A_{\ell,1}(x,t)=1+\sum_{1\leq i\leq \ell}A_{\ell,1}^{(i)}(x,t). 
\end{equation}
(The term $1$ in (\ref{sumA}) comes from the empty word.) 
Also, we have the following system of equations, where to obtain $A_{\ell,1}^{(i)}(x,t)$, we can think of taking words counted by $A_{\ell,1}^{(j)}(x,t)$, $1\leq j\leq \ell$, and adjoining the letter $i$ to the left of them; these functions are then to be multiplied by $xt$ if $|i-j|\leq 1$ (indicating that the length of such words is increased by 1 and one more bond is created), and by $t$ otherwise (to indicate change of the length keeping the number of occurrences of bonds the same); we also need to add $t$ corresponding to the one-letter word $i$.

\begin{small}
\begin{eqnarray*} 
A_{\ell,1}^{(1)}(x,t) &=& t+xtA_{\ell,1}^{(1)}(x,t) + xtA_{\ell,1}^{(2)}(x,t) + t A_{\ell,1}^{(3)}(x,t) + t A_{\ell,1}^{(4)}(x,t) +\cdots  + t A_{\ell,1}^{(\ell)}(x,t);  \\
A_{\ell,1}^{(2)}(x,t) &=& t+xtA_{\ell,1}^{(1)}(x,t) + xtA_{\ell,1}^{(2)}(x,t) + xt A_{\ell,1}^{(3)}(x,t) +t A_{\ell,1}^{(4)}(x,t) + \cdots  + t A_{\ell,1}^{(\ell)}(x,t);  \\
A_{\ell,1}^{(3)}(x,t) &=& t+tA_{\ell,1}^{(1)}(x,t) + xtA_{\ell,1}^{(2)}(x,t) + xt A_{\ell,1}^{(3)}(x,t) +xt A_{\ell,1}^{(4)}(x,t) + \cdots  + t A_{\ell,1}^{(\ell)}(x,t);  \\
&& \hspace{4cm}\vdots \\
A_{\ell,1}^{(\ell)}(x,t) &=& t+tA_{\ell,1}^{(1)}(x,t) + tA_{\ell,1}^{(2)}(x,t) + \cdots + t A_{\ell,1}^{(\ell-2)}(x,t) +xt A_{\ell,1}^{(\ell-1)}(x,t) + xt A_{\ell,1}^{(\ell)}(x,t).  \\
\end{eqnarray*}
\end{small}
Solving the system for the functions $A_{\ell,1}^{(i)}(x,t)$ and applying (\ref{sumA}) we get the desired result.  \end{proof} 

As corollaries to Theorem \ref{thm:1bWords}, we can obtain, e.g. using Mathematica, explicit generating functions for $\ell$ letter alphabets, where $3\leq\ell\leq 7$.  These are presented in Table \ref{tab1}. Note that $A_{1,1}(x,t)$ and $A_{2,1}(x,t)$ are trivial since 
any word $w$ of length $n$ over the alphabet $\{1\}$ or the alphabet 
$\{1,2\}$ has $n-1$ bonds.  We also give expansions of the functions $A_{\ell}(x,t)$ for $\ell=3,...,7$:

\begin{table}
\begin{center}
\begin{tabular}{c|c}
$\ell$ & generating function for distribution of the number of bonds \\
\hline
& \\[-3mm]
$A_{3,1}(x,t)$ & $\frac{1-2(x-1)t-(x-1)^2t^2}{1-t-2  xt-x (x-1) t^2}$ \\[2mm]
\hline
& \\[-3mm]
$A_{4,1}(x.t)$ & $\frac{1-3(x-1)t+(x-1)^2t^2}{1-(3x+1)t+(x^2-1)t^2}$ \\[2mm]
\hline
& \\[-3mm]
$A_{5,1}(x,t)$ & $\frac{1-3(x-1)t+2(x-1)^3t^3}{1-(3x+2)t+2(x-1)t^2+2(x+1)(x-1)^2t^3}$\\[2mm]
\hline
& \\[-3mm]
$A_{6,1}(x,t)$ & $\frac{1-4 (x-1)t+3(x-1)^2t^2+(x-1)^3t^3}{1-2(2x+1)t+(3x^2+2x-5)t^2+(x+1)(x-1)^2t^3}$\\[2mm]
\hline
& \\[-3mm]
$A_{7,1}(x,t)$ & $\frac{1-4(x-1)t+2(x-1)^2t^2+4(x-1)^3t^3-(x-1)^4t^4}{1-(4x+3)t-(7-5x-2x^2)t^2+(4x+5)(x-1)^2t^3-(x+2)(x-1)^3t^4}$
\end{tabular}
\end{center}
\caption{Distribution of the number of bonds on $\ell$-ary words, $\ell=3,...,7$.}\label{tab1}
\end{table}

\begin{tiny}

\begin{eqnarray*}
A_{3,1}(x,t) &=& 1+3 t+(2+7 x) t^2+\left(2+8 x+17 x^2\right) t^3+\left(2+10 x+28 x^2+41 x^3\right) t^4\\
&+& \left(2+12 x+42 x^2+88 x^3+99 x^4\right) t^5+\left(2+14
x+58 x^2+154 x^3+262 x^4+239 x^5\right) t^6\\
&+& \left(2+16 x+76 x^2+240 x^3+524 x^4+752 x^5+577 x^6\right) t^7\\
&+& \left(2+18 x+96 x^2+348 x^3+908 x^4+1692
x^5+2104 x^6+1393 x^7\right) t^8+\cdots;
\end{eqnarray*}

\begin{eqnarray*}
A_{4,1}(x,t) &=& 1+4 t+2 (3+5 x) t^2+2 \left(5+14 x+13 x^2\right) t^3+4 \left(4+17 x+26 x^2+17 x^3\right) t^4\\
&+& 2 \left(13+72 x+162 x^2+176 x^3+89 x^4\right)
t^5\\
&+& 2 \left(21+145 x+422 x^2+662 x^3+565 x^4+233 x^5\right) t^6\\
&+& 4 \left(17+140 x+503 x^2+1016 x^3+1239 x^4+876 x^5+305 x^6\right) t^7\\
&+& 2 \left(55+527
x+2247 x^2+5567 x^3+8717 x^4+8757 x^5+5301 x^6+1597 x^7\right) t^8+\cdots;
\end{eqnarray*}

\begin{eqnarray*}
A_{5,1}(x,t) &=& 1+5 t+(12+13 x) t^2+5 \left(6+12 x+7 x^2\right) t^3+\left(74+222 x+234 x^2+95 x^3\right) t^4\\
&+& \left(184+724 x+1134 x^2+824 x^3+259 x^4\right)
t^5 \\
&+& \left(456+2236 x+4574 x^2+4902 x^3+2750 x^4+707 x^5\right) t^6 \\
&+& \left(1132+6624 x+16800 x^2+23480 x^3+19290 x^4+8868 x^5+1931 x^6\right) t^7 \\
&+& \left(2808+19124
x+57696 x^2+99716 x^3+106666 x^4+71418 x^5+27922 x^6+5275 x^7\right) t^8+\cdots;
\end{eqnarray*}

\begin{eqnarray*}
A_{6,1}(x,t) &=& 1+6 t+4 (5+4 x) t^2+4 \left(17+26 x+11 x^2\right) t^3+2 \left(115+263 x+209 x^2+61 x^3\right) t^4\\
&+& 4 \left(195+590 x+696 x^2+378 x^3+85
x^4\right) t^5 \\
&+& 2 \left(1321+4987 x+7742 x^2+6218 x^3+2585 x^4+475 x^5\right) t^6 \\
&+& 2 \left(4477+20230 x+39031 x^2+41156 x^3+25211 x^4+8534 x^5+1329
x^6\right) t^7 \\
&+& 2 \left(15169+79871 x+183933 x^2+240507 x^3+193107 x^4+95997 x^5+27503 x^6+3721 x^7\right) t^8+\cdots;
\end{eqnarray*}

\begin{eqnarray*}
A_{7,1}(x,t) &=& 1+7 t+(30+19 x) t^2+\left(130+160 x+53 x^2\right) t^3+\left(562+1034 x+656 x^2+149 x^3\right) t^4 \\
&+& \left(2432+5940 x+5598 x^2+2416 x^3+421
x^4\right) t^5 \\
&+& \left(10520+32068 x+39942 x^2+25526 x^3+8400 x^4+1193 x^5\right) t^6 \\
&+& \left(45514+166236 x+257634 x^2+217088 x^3+105512 x^4+28172 x^5+3387
x^6\right) t^7 \\
&+& \left(196898+838274 x+1553178 x^2+1625554 x^3+1039904 x^4+409176 x^5+92190 x^6+9627 x^7\right) t^8+\cdots.
\end{eqnarray*}

\end{tiny}

As noted in the introduction, the number of 
permutations $\sg \in S_n$ such that 
$\bx[\sg]=0$ equals the  number of 
permutations $\sg \in S_n$ such that 
$\bond[\sg]=0$.  The same applies to words. Thus, plugging in $x=0$ in the functions in Table \ref{tab1}, one gets generating functions for avoidance of the 1-box pattern (alternatively, we can plug in $x=0$ in the matrix $\mathbb{A}_{\ell,1}$ in Theorem \ref{thm:1bWords} to get the most general case and to work out particular small values of $\ell$); in Table \ref{tab2}, we list initial values of the respective sequences indicating connections to the OEIS \cite{oeis}. In particular, the connection to the sequence A118649 led us to solving a conjecture of Mathar (published in \cite[A118649]{oeis}) to be discussed in Subsection \ref{subLEGO}.

\begin{table}
\begin{center}
\begin{tabular}{c|c}
$\ell$ & generating function for permutations which avoid the 1-box pattern \\
\hline
& \\[-3mm]
$A_{3,1}(0,t)$ & $\frac{1+2t -t^2}{1-t}$ \\[2mm]
\hline
& \\[-3mm]
$A_{4,1}(0.t)$ & $\frac{1+3t+t^2}{1-t-t^2}$ \\[2mm]
\hline
& \\[-3mm]
$A_{5,1}(0,t)$ & $\frac{1+3t-2t^3}{1-2t-2t^2+2t^3}$\\[2mm]
\hline
& \\[-3mm]
$A_{6,1}(0,t)$ & $\frac{1+4t+3t^2-t^3}{1-2t-5t^2+t^3}$\\[2mm]
\hline
& \\[-3mm]
$A_{7,1}(x,t)$ & $\frac{1+4t+2t^2-4t^3-t^4}{1-3t-7t^2+5t^3+2t^4}$
\end{tabular}
\end{center}
\caption{Distribution of $\ell$-ary words which avoid 
the 1-box pattern for  $\ell=3,...,7$.}\label{tab1.1}
\end{table}

\begin{table}
\begin{center}
\begin{tabular}{c|c|c}
$\ell$ & number of $\ell$-ary words avoiding the 1-box pattern & sequence in \cite{oeis} \\
\hline
3 & 1, 3, 2, 2, 2, 2, 2, 2, 2, 2, ... & \\
\hline
4 & 1, 4, 6, 10, 16, 26, 42, 68, 110, 178, ... & A006355, $n\geq 1$ \\
\hline
5 & 1, 5, 12, 30, 74, 184, 456, 1132, 2808, 6968, ... & A118649, $n\geq 1$ \\
\hline
6 & 1, 6, 20, 68, 230, 780, 2642, 8954, 30338, 102804, ... & \\
\hline
7 & 1, 7, 30, 130, 562, 2432, 10520, 45514, 196898, 851828, ... & \\
\end{tabular}
\end{center}
\caption{Avoidance of the 1-box patterns in $\ell$-ary words for lengths $n$ up to 9.}\label{tab2}
\end{table}

In \cite{KMMP}, Knopfmacher, Mansour, Munagi, and Prodinger 
studied generating functions for smooth $\ell$ words where 
a word $w = w_1 \ldots w_n \in [\ell]^n$ is smooth if $|w_i-w_{i+1}| \leq 1$ 
for $1 \leq i < n$. Thus in our notation, 
$w\in [\ell]^n$ is smooth if $\bond[w] =n-1$.  Let 
$M_{n,1,\ell}$ denote the number of $w \in [\ell]^n$ such 
that $\bond[w] = n-1$ and 
$sm_\ell(t) = 1+ \sum_{n \geq 1} M_{n,1,\ell}t^n$.  Then 
 Knopfmacher, Mansour, Munagi, and Prodinger \cite[Theorem 2.2]{KMMP} proved 
that 
\begin{equation}
sm_\ell(t) = 1+\frac{t(\ell -(3\ell +2)t)}{(1-3t)^2} + 
\frac{2t^2}{(1-3t)^2}\frac{1+U_{\ell -1}\left(\frac{1-t}{2t}\right)}{U_{\ell}\left(\frac{1-t}{2t}\right)}
\end{equation}
where $U_r(t)$ is the Chebyshev polynomial of the second kind defined 
by 
$$U_r(\cos(\theta)) = \frac{\sin((r+1)\theta)}{\sin(\theta)}.$$
Alternatively, one can define the polynomials by 
recursion by setting $U_0(t) =1$, $U_1(t) =2t$, $U_2(t) =4t^2 -1$, and 
$$U_r(t) =2tU_{r-1}(t) -U_{r-2}(t) \  \mbox{for} \ r \geq 3.$$

We can obtain the same generating functions from 
our generating function $B_{\ell,1}(x,t)$. That is, 
clearly 
$$B_{\ell,1}(1/x,xt) = 1 + \sum_{n \geq 1} \sum_{w \in [\ell]^n} 
x^{n -\bond[w]}t^n$$
so that 
$$C_{\ell,1}(x,t):= \frac{1}{x}(B_{\ell,1}(1/x,xt)-1) = 
\sum_{n \geq 1} \sum_{w \in [\ell]^n} 
x^{n -1-\bond[w]}t^n.$$
Hence 
\begin{equation*}
sm_\ell(t) = 1+ C_{\ell,1}(0,t).
\end{equation*}

\subsection{Distribution of 1-box patterns over words.}\label{1boxDist}

One can use similar methods to find 
the distribution of $\bx[w]$ for $w \in [\ell]^*$. In this 
case we have to keep track of more information. This 
is due to the fact that extra contribution to $\bx[w]$ caused 
by adding an extra letter at the front of a word $w$ depends on the first 
two letters of $w$.  For example, $\bx[12] =x^2t^2$ and $\bx[112] = x^3t^3$ 
so that adding 1 to the front of $w =12$ increased $x^{\bx[w]}t^{|w|}$ by 
a factor of $xt$.  However, $\bx[13] =t^2$ and $\bx[113] = x^2t^3$ 
so that adding 1 to the front of $w =13$ increased $x^{\bx[w]}t^{|w|}$ by 
a factor of $x^2t$.

For $1 \leq i,j \leq \ell$, let  
$$B_{\ell,1}^{(ij)} = \sum_{w \in ij[\ell]^*} WT(w)$$ 
where $WT(w) = x^{\bx[w]}t^{|w|}$ and $ij[\ell]^*$ denotes the set of words over $[\ell]$ that begin with letters $ij$.
For any statement $S$, let $\chi(S) = 1$ if $S$ is true and 
$\chi(S) =0$ if $S$ is false. Then we claim that 
for all $1 \leq i,j \leq \ell$, 
\begin{eqnarray}\label{1boxwdrec}
&&B_{\ell,1}^{(ij)}(x,t) = x^{2\chi(|i-j]\leq 1)}t^2 + \\
&&\sum_{k=1}^{\ell} 
(t\chi(|i-j| > 1) + xt\chi(|i-j| \leq 1)\chi(|j-k| \leq 1) + \nonumber \\
&& \ \ \ \ \ \ \ \ \ \ \ x^2t
\chi(|i-j| \leq 1)\chi(|j-k| > 1))B^{(jk)}_{\ell,1}(x,t). \nonumber 
\end{eqnarray}
That is, the words in $ij[\ell]^*$ are of the 
form $ij$ plus words $ijkv$ where $k \in [\ell]$ and 
$v \in [\ell]^*$. Now 
\begin{equation*}
WT[ij] = \begin{cases} t^2 & \mbox{ if } |i-j|>1  \mbox{ and} \\ 
x^2t^2 & \mbox{ if } |i-j|\leq 1.
\end{cases}
\end{equation*}
Similarly, 
\begin{equation*}
WT[ijkv] = \begin{cases} tWT[jkv] & \mbox{ if } |i-j|>1,  \\ 
xt WT[jkv] & \mbox{ if } |i-j|\leq 1 \mbox{ and }  |j-k|\leq   1,  \mbox{ and}\\x^2t WT[jkv] & \mbox{ if } |i-j|\leq 1 \mbox{ and }  |j-k|>  1.
\end{cases}
\end{equation*}

The set of equations of the form (\ref{1boxwdrec}) 
can be written out in matrix form. That is, 
let 
$\vec{B}_{\ell,1}$ be the row vector of length $\ell^2$ of
the $B^{(ij)}_{\ell,1}(t,x)$ where the elements are listed in 
the lexicographic order of the pairs $(ij)$. For example,  $\vec{B}_{3}$ 
equals 
{\tiny $$(B_{3,1}^{(11)}(x,t),B_{3,1}^{(12)}(x,t),B_{3,1}^{(13)}(x,t),B_{3,1}^{(21)}(x,t), 
B_{3,1}^{(22)}(x,t),B_{3,1}^{(23)}(x,t),B_{3,1}^{(31)}(x,t),B_{3,1}^{(32)}(x,t), B_{3,1}^{(33)}(x,t)).$$}
Similarly, let 
$\vec{I}_{\ell,1}$ be the row vector of length $\ell^2$ of the 
terms $t^2x^{2\chi(|i-j| \leq 1)}$ again listed in 
the lexicographic order on the pairs $ij$. For example, 
$$ \vec{I}_{3,1}=(x^2t^2,x^2t^2,t^2,x^2t^2,x^2t^2,x^2t^2,t^2,x^2t^2,x^2t^2).$$
Then one can write  a set of equations of the form (\ref{1boxwdrec}) in 
the form 
$$(\vec{I}_{\ell,1})^T = \mathbb{B}_{\ell,1} (\vec{B}_{\ell,1})^T$$
where 
$ \mathbb{B}_{\ell,1}$ is an $\ell^2\times \ell^2$ matrix. 
For example, $\mathbb{B}_{3,1}$ is the matrix

$$
\begin{pmatrix}
xt-1 & xt & x^2t & 0 & 0 & 0 & 0 & 0 & 0 \\
0  & -1 & 0  & xt & xt  & xt & 0 & 0 & 0 \\
0  & 0 & -1  & 0  & 0  & 0 & t & t & t \\
xt & xt & x^2t & -1 & 0 & 0 & 0 & 0 & 0 \\
0  & 0 & 0  & xt & xt-1  & xt & 0 & 0 & 0 \\
0  & 0 & 0  & 0  & 0  & -1 & x^2t & xt & xt \\
t & t & t & 0 & 0 & 0 & -1 & 0 & 0 \\
0 & 0  & 0 & 0  & xt & xt  & xt & -1 & 0 \\
0  & 0 & 0  & 0  & 0  & 0 & x^2t & xt & xt -1 \\
\end{pmatrix}.$$

Note that since 
setting $x=t=0$ in $ \mathbb{B}_{\ell,1}$ will gives 
an $\ell \times \ell$ diagonal matrix with $-1$s on 
the diagonal,  $ \mathbb{B}_{\ell,1}$ is invertible. Thus 
$$(\vec{B}_{\ell,1})^T = \mathbb{B}_{\ell,1}^{-1} (\vec{I}_{\ell,1})^T.$$
Let $\vec{1}_{\ell,1}$ denote the vector of length $\ell^2$ consisting 
of all 1s. Then 
$$\sum_{1 \leq i,j \leq \ell} B^{(ij)}_{\ell,1}(x,t) = 
\vec{1}_{\ell,1} \mathbb{B}_{\ell,1}^{-1} (\vec{I}_{\ell,1})^T.$$
Taking into account the empty word and all the 
words of length 1 will yeild the following theorem. 

\begin{theorem}\label{thm:1boxgf} For all $\ell \geq 2$, 
\begin{equation*}
B_{\ell,1}(x,t) = 1 + \ell t + \vec{1}_{\ell,1} 
\mathbb{B}_{\ell,1}^{-1} (\vec{I}_{\ell,1})^T.
\end{equation*}
\end{theorem}

We have used Theorem \ref{thm:1boxgf} to compute 
$B_{\ell,1}(x,t)$ for $\ell =3,4, \ \mbox{and } 5$.

\begin{equation*}\label{eq:B3}
B_{3,1}(x,t)= \frac{1+2(1-x)t-(1+4x-5x^2)t^2+2x(1-x)^2t^3+x^2(1-x)^2t^4}{1-(1+2x)t+2x(1-x)t^2+x^2(1-x)t^3};
\end{equation*}

\begin{equation*}\label{eq:B4}
B_{4,1}(x,t) = \frac{1+3 (1-x)t +  \left(1-9 x+8 x^2\right)t^2 -3x(1-x)^2 t^3
+ x^2(1-x)^2 t^4}{1-(1+3x)t- \left(1-3 x+2 x^2\right) t^2 - x \left(3-4
x+x^2\right)t^3 - x^2(1-x)^2 t^4};
\end{equation*}

\begin{equation*}\label{eq:B5}
B_{5,1}(x,t) = \frac{f_{5,1}(x,t)}{g_{5,1}(x,t)}
\end{equation*}
where 
\begin{eqnarray*}
f_{5,1}(x,t) &=& 1+3  (1-x)t+9x(1-x) t^2 - 2(1-x)^2 (1+2 x)t^3 +\\
&&6 x (1-x)^2  (1+x)t^4-4(1-x)^3 x^3t^6 
\end{eqnarray*}
and 
\begin{eqnarray*}
g_{5,1}(x,t) &=& 1 - (2+3 x)t -\left(2-6 x+4 x^2\right)t^2 - 
\left(-2-6 x+8 x^2\right)t^3 -\\
&& 6(1-x)^2 x (1+x) t^4 -4 (1-x)^2 x^3t^5+ 4  (1-x)^3 x^3t^6. 
\end{eqnarray*}

Using the generating functions above, we have computed 
some of the initial terms in their Taylor series expansions. 

\begin{tiny}

\begin{eqnarray*}
&&B_{3,1}(x,t) = 1+3t+ (2+7 x^2) t^2 + (2+8 x^2+17 x^3)t^3+\\
&&(2+10 x^2+20 x^3+49 x^4)t^4 + (2+12 x^2+26 x^3+64 x^4+139 x^5)t^5 +\\
&&(2+14 x^2+32 x^3+88 x^4+200 x^5+393 x^6)t^6 +\\
&&(2+16 x^2+38
x^3+114 x^4+290 x^5+614 x^6+1113 x^7)t^7+ \\
&&(2+18 x^2+44 x^3+142 x^4+392 x^5+932 x^6+1880 x^7+3151 x^8)t^8 + \cdots. 
\end{eqnarray*}

\begin{eqnarray*}
&&B_{4,1}(x,t) =1+4t +(6+10 x^2)t^2+(10+28 x^2+26 x^3)t^3 + \\
&&(16+68 x^2+72 x^3+100 x^4)t^4+(26+144 x^2+174 x^3+338 x^4+342 x^5)t^5 + \\
&&(42+290 x^2+368 x^3+930 x^4+1256 x^5+1210x^6) t^6 + \\
&&(68+560 x^2+740 x^3+2232 x^4+3612 x^5+4932 x^6+4240 x^7)t^7+ \\
&&(110+1054 x^2+1428 x^3+4996 x^4+8984 x^5+15246 x^6+18820 x^7+14898 x^8)t^8 
+ \cdots,
\end{eqnarray*}

\begin{eqnarray*}
&&B_{5,1}(x,t) = 1+5 t+\left(12+13 x^2\right) t^2+5 \left(6+12 x^2+7 x^3\right) t^3+\\
&&\left(74+222 x^2+160 x^3+169 x^4\right) t^4+\left(184+724 x^2+592 x^3+974
x^4+651 x^5\right) t^5+\\
&&\left(456+2236 x^2+1932 x^3+4238 x^4+4048 x^5+2715 x^6\right) t^6+\\
&&\left(1132+6624 x^2+5968 x^3+16036 x^4+18372 x^5+18982 x^6+11011
x^7\right) t^7+\\
&&\left(2808+19124 x^2+17688 x^3+56072 x^4+71724 x^5+94282 x^6+83828 x^7+45099 x^8\right) t^8+ \cdots .
\end{eqnarray*}

\end{tiny}

\subsection{Solving three conjectures of Hardin.}\label{threeHardin}

Note that 
$$\overline{B}_{k,1}(x,t) := B_{k,1}[1/x,xt] = 
\sum_{w \in [k]^*} x^{n-(\bx[w])}t^n$$ 
so that $\overline{B}_{k,1}(0,t)$ is the generating function 
of all words in $w =w_1 \ldots w_n \in[k]^*$ such that 
 $\bx[w]=n$, i.e. each letter of $w$ differs from at least one neighbor by 1 or less. We have computed $\overline{B}_{k,1}(0,t)$ for 
$k = 3,4,5$. 
\begin{equation*}
\overline{B}_{3,1}(0,t) = \frac{1-2t+5t^2+2t^3+t^4}{1-2t-2t^2 -t^3}.
\end{equation*}
The initial terms of this series are 
$1,0,7,17,49,139,393,1113,3151,8921, \ldots$. This is the sequence 
A221591 which was apparently computed directly from its combinatorial 
definition 
by R. H. Hardin.  If  $\overline{B}_{3,1}(0,t) = \sum_{n \geq 0} 
b_{3,1,n}t^n$, then Hardin observed empirically that 
$b_{3,1,n} = 2b_{3,1,n-1}+  2b_{3,1,n-2}+ b_{3,1,n-3}$ for 
$n > 4$. This recursion follows immediately from the 
generating function for  $\overline{B}_{3,1}(0,t)$ so that we have 
proved Hardin's conjecture. 

\begin{equation*}
\overline{B}_{4,1}(0,t) = \frac{1-3t+8t^2-3t^3+t^4}{1-3t-2t^2+t^3-t^4}.
\end{equation*}
The initial terms of this series are 
$1,0,10,26,100,342,1210,4240,14898,52306,\ldots$. This is the sequence 
A221569 which was also computed directly form its combinatorial definition 
by R. H. Hardin.  If  $\overline{B}_{4,1}(0,t) = \sum_{n \geq 0} 
b_{4,1,n}t^n$, then Hardin observed empirically that 
$b_{4,1,n} = 3b_{4,1,n-1}+  2b_{4,1,n-2}- b_{4,1,n-3} + b_{4,1,n-4}$ for 
$n > 5$. This recursion follows immediately from the 
generating function for  $\overline{B}_{4,1}(0,t)$ so that we have 
also proved this conjecture of Hardin.

\begin{equation*}
\overline{B}_{5,1}(0,t) = \frac{1-3t+9t^2-4t^3+6t^4+4t^6}{1-3t-4t^2-6t^4-4t^5-4t^6}.
\end{equation*}
The initial terms of this series are 
$1,0,13,35,169,651,2715,11011,45099,184063,\ldots$. This is the sequence 
A221592 which was also computed directly form its combinatorial definition 
by R. H. Hardin.  If  $\overline{B}_{5,1}(0,t) = \sum_{n \geq 0} 
b_{5,1,n}t^n$, then Hardin observed empirically that 
$b_{4,1,n} = 3b_{4,1,n-1}+  4b_{4,1,n-2}+ 6b_{4,1,n-4} + 6b_{4,1,n-5} + 5b_{4,1,n-6}$ for 
$n > 6$. This recursion follows immediately from the 
the generating function for  $\overline{B}_{5,1}(0,t)$ so that we have 
also proved this conjecture of Hardin.

\subsection{Solving an enumerative conjecture on LEGO.}\label{subLEGO}

A ``stable LEGO wall'' is a wall in which seams do not match up from one level to the next. Stable LEGO walls of width 7 and heights 1 and 2 when using bricks of length 2, 3, and 4 can be found in Figure \ref{LEGO} (the numbers should be ignored there for the moment).

\begin{figure}[h]
\begin{center}
\includegraphics[scale=0.4]{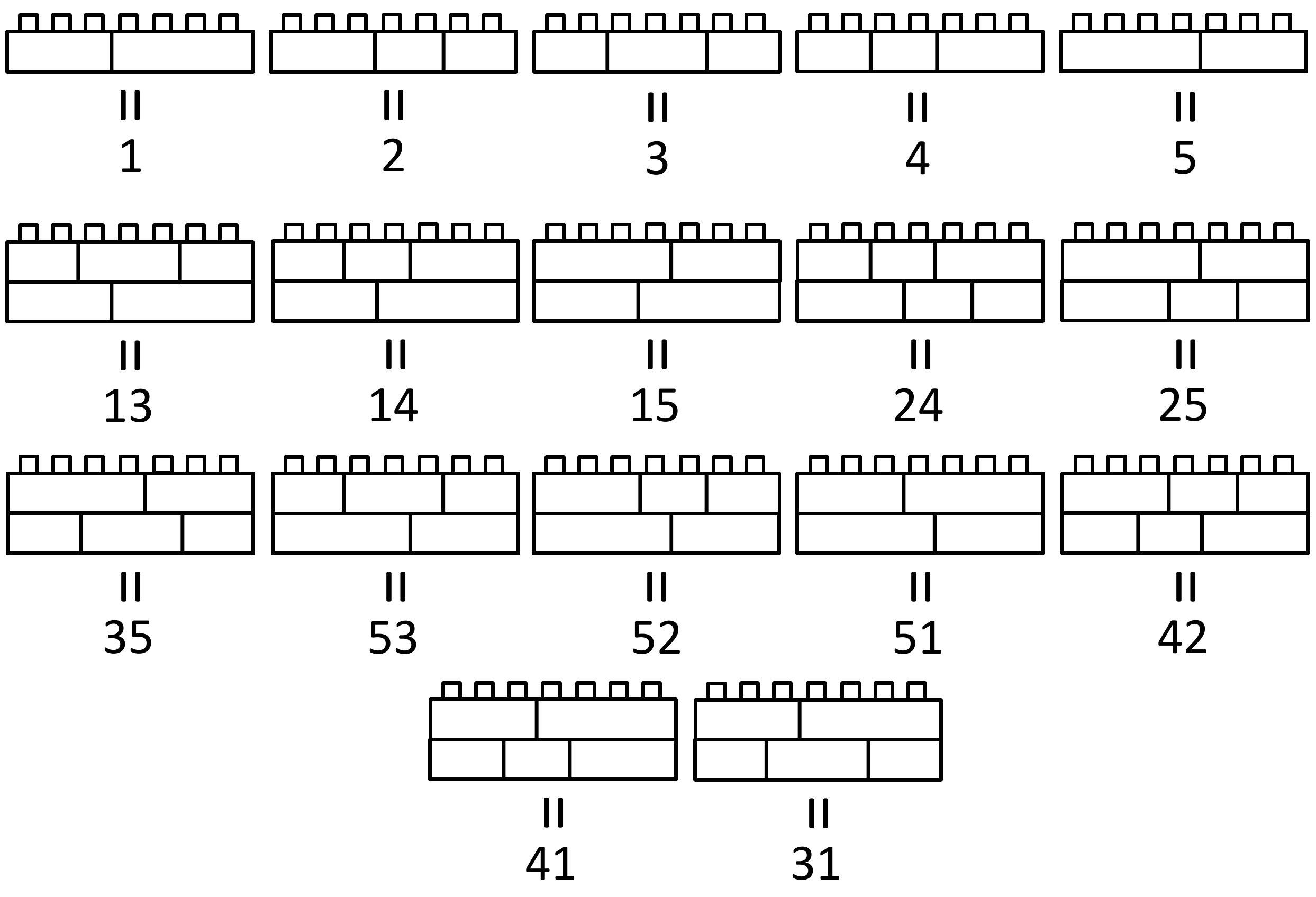}
\caption{Stable LEGO walls of width 7 and heights 1 and 2.}\label{LEGO}
\end{center}
\end{figure}

\begin{lemma}\label{bij1} There is a bijection between words over the alphabet $A=\{1,2,3,4,5\}$ of length $n$ that avoid the $1$-box pattern and stable LEGO walls of width $7$ and height $n$ when using bricks of length $2$, $3$, and $4$.\end{lemma}

\begin{proof} Encode the eligible LEGO configurations of height 1 by the elements of $A$ as shown in Figure \ref{LEGO}, which gives a bijection between the objects in the case of $n=1$. 

More generally, given a word $w=w_1w_2\ldots w_n$ avoiding the 1-box pattern, we let the $i$-th level from below of the wall corresponding to $w$ be given by the configuration corresponding to the letter $w_i$ defined in Figure~\ref{LEGO}. For example, the correspondence for the case $n=2$ is shown in Figure~\ref{LEGO}. 

It is straightforward to check that the prohibited factors of words, namely 12, 23, 34, 45, 54, 43, 32, and 21, correspond to the prohibited configurations in LEGO, and vice versa.\end{proof}

Using Lemma \ref{bij1}, the function corresponding to $\ell=5$ and $x=0$ in Table \ref{tab1}, and taking care of the offset (removing the number 2 in the sequence \cite[A118649]{oeis} and shifting down the indices of the larger numbers), we can confirm a conjecture of R. J. Mathar that stable LEGO walls satisfying the assumptions of Lemma \ref{bij1} are counted by the following generating function:
$$\frac{1+3t-2t^3}{1-2t-2t^2+2t^3}.$$

\subsection{$(1,k)$-rectangle patterns for $k\geq 2$; solving a conjecture of Barker and enumerating two sequences of Hardin.} \label{secWords2}

Given a word $w =w_1 \ldots w_n \in [\ell]^n$, let 
$k\mbox{-}\bond[w] = |\{i: |w_i-w_{i+1}| \leq k\}|$. 
It is straightforward to generalize Theorems \ref{thm:1bWords}  and 
\ref {thm:1boxgf} to find the distribution of $k\mbox{-}\bond[w]$ and 
$(1,k)\mbox{-rec}(w)$, the number of $(1,k)$-rectangle patterns in $w$, over words $w$ in $[\ell]^*$. 
That is, we  claim 
that the  same method of proof can also be used to find 
the generating function 
$$A_{\ell,k} = \sum_{w \in [\ell]^*} x^{k\mbox{-}\bond[w]} t^{|w|} = 
\sum_{m,n \geq 0} a_{\ell,k}(m,n)x^m t^n$$
for $k \geq 2$. 
Thus $a_{\ell,k}(m,n)$ is the number of words $w \in [\ell]^n$ 
such that $k\mbox{-}\bond[w] =m$.  

Let 
$\mathbb{A}_{\ell,k}$ be the $\ell \times \ell$ matrix 
whose entries on the main diagonal consists of all $xt-1$'s, whose 
entries on the first $k$ superdiagonals and the first $k$ 
subdiagonals are $xt$, and whose remaining entries are $t$. 
Then we have the following theorem. 

\begin{theorem} \label{thm:kbonds}
For all 
$\ell,k \geq 1$, 
$$A_{\ell,k}(x,t) = 1+(\underbrace{1,\ldots,1}_{\ell})\mathbb{A}_{\ell,k}^{-1}(\underbrace{-t,\ldots,-t}_{\ell})^{T}.$$
\end{theorem}
\begin{proof} 
For $i =1, 2,\ldots, \ell$, let 
$$A^{(i)}_{\ell,k}(x,t) = \sum_{w \in i[\ell]^*} x^{k\mbox{-}\bond[w]} t^{|w|} = 
\sum_{m,n \geq 0} a^{(i)}_{\ell,k}(m,n)x^m t^n.$$
When $k \geq 2$, we can follow the proof of Theorem \ref{thm:1bWords} and 
find simple recurrences for the functions $A^{(i)}_{\ell,k}(x,t)$. 
Indeed, in this case we may have more possibilities to create an occurrence of the $k$-box pattern while adjoining letter $i$ from the left side, so that in the terminology of the proof of Theorem \ref{thm:1bWords}, 
\begin{eqnarray*}
A_{\ell,k}^{(i)}(x,t) &=& t+tA_{\ell,k}^{(1)}(x,t) + \cdots + tA_{\ell,k}^{(i-k-1)}(x,t) + \\
&&xt A_{\ell}^{(i-k)}(x,t) + \cdots +xt A_{\ell,k}^{(i+k)}(x,t) + \\
&&t A_{\ell,k}^{(i+k+1)}(x,t)+\cdots  + t A_{\ell}^{(\ell)}(x,t).
\end{eqnarray*}
 Thus, for an arbitrary $k$, the first row in the matrix $A$ in Theorem \ref{thm:1bWords} is the vector 
$$(xt-1,\underbrace{xt,\ldots,xt}_{k},t,\ldots,t),$$ the second row is the vector $$(xt,xt-1,\underbrace{xt,\ldots,xt}_{k},t,\ldots,t),$$ and, more generally, any row in $A$ in this case is of the form $$(t,\ldots,t,\underbrace{xt,\ldots,xt}_{k},xt-1,\underbrace{xt,\ldots,xt}_{k},t,\ldots,t).$$
\end{proof}

For example, the generating function $A_{\ell,2}(x,t)$ is equal to $$1+(\underbrace{1,\ldots,1}_{\ell})A_{\ell,2}^{-1}(\underbrace{-t,\ldots,-t}_{\ell})^{T}$$ where $A_{\ell,2}$ is the following $\ell\times\ell$ matrix: 

$$\mathbb{A}_{\ell,2}=\left(\begin{array}{ccccccccccc} 
   xt-1 & xt         &     xt      &       t   & t  &t  &t       & \cdots &t & t & t \\ 
   xt     & xt-1     &    xt     &       xt   & t   &t   &t      & \cdots &t & t & t \\ 
   xt       & xt         &    xt-1 &       xt & xt  &t  &t        & \cdots &t & t & t \\ 
 t       &  xt         &    xt &       xt-1 & xt   & xt &t       & \cdots &t & t & t \\ 
\vdots & \vdots & \vdots & \vdots & \vdots & \vdots & \vdots & \ddots & \vdots 
& \vdots & \vdots \\ 
    t      &       t    &     t      &    t   &t  &t   &  t & \cdots  & xt&     xt    &  xt-1 \end{array}\right).$$

We have used Theorem \ref{thm:kbonds} to compute 
the generating functions $A_{\ell,2}(x,t)$ for $\ell = 4,5,6, 7$.

\begin{table}
\begin{center}
\begin{tabular}{c|c}
$\ell$ & generating function  $A_{\ell,2}(x,t)$ for $\ell = 4,5,6,7$.\\
\hline
& \\[-3mm]
 4 & $\frac{1-3 t (-1+x)-2 t^2 (-1+x)^2}{1-t-3 t x-2 t^2 (-1+x) x}$ \\[2mm]
\hline
& \\[-3mm]
5 & $\frac{1-4 t (-1+x)+t^3 (-1+x)^3}{1+t^2 (-1+x)+t^3 (-1+x)^2 x-t (1+4 x)}$\\[2mm]
\hline
& \\[-3mm]
 6 & $\frac{1-4 t (-1+x)-t^2 (-1+x)^2+t^3 (-1+x)^3}{1-t^2 (-1+x)^2+t^3 (-1+x)^2 (1+x)-2 t (1+2 x)}$\\[2mm]
\hline
& \\[-3mm]
 7 & $\frac{1-5 t (-1+x)+2 t^2 (-1+x)^2+4 t^3 (-1+x)^3-2 t^4 (-1+x)^4}{1-2 t^4 (-1+x)^3 (1+x)+2 t^3 (-1+x)^2 (1+2 x)-t (2+5 x)+2 t^2
\left(-2+x+x^2\right)}$
\end{tabular}
\end{center}
\caption{Distribution of the $2\mbox{-bond}$ on $\ell$-ary words, $\ell=4,5,6,7$.}\label{tab12}
\end{table}

\begin{tiny}

\begin{eqnarray*}
A_{4,2}(x,t) &=& 1+4 t+2  (1+7 x)t^2+2  \left(1+6 x+25 x^2\right)t^3+2  \left(1+7 x+31 x^2+89 x^3\right)t^4+ \\
&&2 \left(1+8 x+42 x^2+144 x^3+317 x^4\right) t^5 +2
 \left(1+9 x+54 x^2+222 x^3+633 x^4+1129 x^5\right)t^6+\\
&&2  \left(1+10 x+67 x^2+316 x^3+1095 x^4+2682 x^5+4021 x^6\right)t^7+\\
&&2  \left(1+11 x+81
x^2+427 x^3+1707 x^4+5145 x^5+11075 x^6+14321 x^7\right)t^8+\cdots;
\end{eqnarray*}

\begin{eqnarray*}
A_{5,2}(x,t) &=& 1+5 t+ (6+19 x)t^2+5  \left(2+8 x+15 x^2\right)t^3+ \left(16+88 x+226 x^2+295 x^3\right)t^4+ \\
&& \left(26+176 x+606 x^2+1156 x^3+1161 x^4\right)t^5+
\left(42+342 x+1428 x^2+3644 x^3+5600 x^4+4569 x^5\right)t^6+\\
&& \left(68+644 x+3170 x^2+9840 x^3+20250 x^4+26172 x^5+17981 x^6\right)t^7+\\
&& \left(110+1190
x+6708 x^2+24456 x^3+61446 x^4+106686 x^5+119266 x^6+70763 x^7\right)t^8
+\cdots;
\end{eqnarray*}

\begin{eqnarray*}
A_{6,2}(x,t) &=& 1+6 t+12  (1+2 x)t^2+4 \left(7+22 x+25 x^2\right) t^3+ \left(62+294 x+522 x^2+418 x^3\right)t^4+\\
&&4 \left(35+214 x+552 x^2+706 x^3+437x^4\right)t^5+
2 \left(157+1191 x+3926 x^2+7154 x^3+7245 x^4+3655 x^5\right)t^6+\\
&& \left(706+6364 x+25702 x^2+59624 x^3+85166 x^4+71804 
x^5+30570 x^6\right)t^7+\\
&&2
 \left(793+8295 x+39525 x^2+111571 x^3+202491 x^4+239637 x^5+173575 x^6+63921 x^7\right)t^8+\cdots;
\end{eqnarray*}

\begin{eqnarray*}
A_{7,2}(x,t) &=& 1+7 t+ (20+29 x)t^2+ \left(62+156 x+125 x^2\right)t^3+
 \left(186+710 x+962 x^2+543 x^3\right)t^4+\\
&& \left(566+2820 x+5658 x^2+5400 x^3+2363 x^4\right)t^5+ 
\left(1712+10648 x+27710 x^2+38526 x^3+28766 x^4+10287 x^5\right)t^6+\\
&& \left(5192+38520 x+124086 x^2+222928 x^3+239930 x^4+148100 x^5+44787
x^6\right)t^7+\\
&& \left(15728+135852 x+519888 x^2+1149548 x^3+1594738 x^4+1409754 x^5+744298 x^6+194995 x^7\right)t^8+\cdots.
\end{eqnarray*}

\end{tiny}

\begin{table}
\begin{center}
\begin{tabular}{c|c}
$\ell$ & generating function for permutations which avoid the $(1,2)$-rectangle pattern \\
\hline
& \\[-3mm]
$A_{4,2}(0,t)$ & $\frac{1+3t-2t^2}{1-t}$ \\[2mm]
\hline
& \\[-3mm]
$A_{5,2}(0,t)$ & $\frac{1+4t-t^3}{1-t-t^2}$ \\[2mm]
\hline
& \\[-3mm]
$A_{6,2}(0,t)$ & $\frac{1+4t-t^2-t^3}{1-2t-t^2+t^3}$\\[2mm]
\hline
& \\[-3mm]
$A_{7,2}(0,t)$ & $\frac{1+5t+2t^2-4t^3-2t^4}{1-2t-4t^2+2t^3+2t^4}$
\end{tabular}
\end{center}
\caption{Enumeration of $\ell$-ary words which avoid 
the $(1,2)$-rectangle pattern for  $\ell=4,5,6,7$.}\label{tab1.12}
\end{table}

Clearly the number of words $w \in [\ell]^n$ such 
that $k\mbox{-}\bond[w]=0$ equals the number of words $w \in [\ell]^n$ such 
that $(1,k)\mbox{-rec}(w)=0$. 
Plugging in $x=0$ in the functions in Table \ref{tab12} one gets generating functions for avoidance of the $(1,2)$-rectangle pattern. In Table \ref{tab22},  
we list initial values of the respective sequences indicating connections to the OEIS \cite{oeis}.

\begin{table}
\begin{center}
\begin{tabular}{c|c|c}
$\ell$ & number of $\ell$-ary words avoiding the (1,2)-rectangle pattern & sequence in \cite{oeis} \\
\hline
4 & 1, 4, 2, 2, 2, 2, 2, 2, 2, 2, ... & \\
\hline
5 & 1, 5, 6, 10, 16, 26, 42, 68, 110, 178, ... & A006355, $n\geq 2$ \\
\hline
6 & 1, 6, 12, 28, 62, 140, 314, 706, 1586, 3564, ... & A052994, $n \geq 2$\\
\hline
7 & 1, 7, 20, 62, 186, 566, 1712, 5192, 15728, 47688, ... & \\
\end{tabular}
\end{center}
\caption{Avoidance of the (1,2)-rectangle patterns in $\ell$-ary words for lengths $n$ up to 9.}\label{tab22}
\end{table}

We note that the sequence A052994 has no combinatorial interpretation in 
the OEIS so now we have given a combinatorial interpretation to 
this sequence. Also, comparing Tables \ref{tab2} and \ref{tab22}, and using an interpretation of \cite[A006355]{oeis}, one has the truth of the following proposition that we explain combinatorially.

\begin{proposition} For $n\geq 2$, the following objects are equinumerous:
\begin{itemize}
\item[(i)] words of length $n$ over the alphabet $[5]$ that avoid the $(1,2)$-rectangle pattern;
\item[(ii)] words of length $n$ over the alphabet $[4]$ that avoid the $1$-box pattern;
\item[(iii)] binary words of length $n+3$ that contain no singletons, that is, any $0$ has a $0$ staying next to it, and any $1$ has a $1$ staying next to it.
\end{itemize}
Thus, according to \cite[A006355]{oeis}, any of these objects is counted by $F_{n-1}+F_{n+2}$ where $F_{n}$ is the $n$th Fibonacci number defined as $F_0=F_1=1$ and $F_n=F_{n-1}+F_{n-2}$.
\end{proposition}

\begin{proof} Equinumeration of (i) and (ii) follows directly from the observation that the letter 3 never appears in words described by (i), so that we can take any such word, make the substitution of letters $4\rightarrow 3$ and $5\rightarrow 4$ to get a proper word described by (ii); this operation is clearly reversible. 

Equinumeration of (ii) and (iii) is established by the following bijective map from (ii) to (iii). Let a word $w=w_1w_2\ldots w_n$ described by (ii) is given and we want to obtain its binary image $u=u_1u_2\ldots u_{n+3}$. If $w_1\in\{1,2\}$ then $u_1u_2=00$; if $w_1\in\{3,4\}$ then $u_1u_2=11$. Also, no matter what $u_{n+2}$ is, we set $u_{n+3}=u_{n+2}$. To recover the letters $u_3, u_4, \ldots, u_{n+2}$, we read $w$ from left to right letter by letter: if $w_i\in\{1,4\}$, then $u_{i+2}=u_{i+1}$;  if $w_i\in\{2,3\}$, then $u_{i+2}\neq u_{i+1}$. For example, the word 3413142 avoiding the 1-box pattern is mapped to 1100011100. In Table \ref{tabMap}, we provide our map for all words of length $n=2,3$. 

We do not provide a proper proof of the fact that the map described by us from (ii) to (iii) is a bijection just giving a couple of remarks why this is the case. Indeed, if $w_i\in\{2, 3\}$ and $i<n$ then $w_{i+1}\in\{1,4\}$ and thus $u_{i+2}=u_{i+3}$. This, together with the fact that $u_{n+3}=u_{n+2}$ makes sure that $u$ has no singletons. \end{proof}

\begin{table}
\begin{center}
\begin{tabular}{|c|c|c|c|c|}
\hline
13 & 00011 & \ & 131 & 000111\\
\hline 
14 & 00000 & \ & 141 & 000000 \\
\hline 
24 & 00111 & \ & 142 & 000011 \\
\hline 
31 & 11000 & \ & 241 & 001111\\
\hline 
41 & 11111 & \ & 242 & 001100 \\
\hline 
42 & 11100 & \ & 313 & 110011\\
\hline 
& & \ & 314 & 110000 \\
\hline 
& & \ & 413 & 111100 \\
\hline 
& & \ & 414 & 111111 \\
\hline 
& & \ & 424 & 111000 \\
\hline 
\end{tabular}
\end{center}
\caption{Mapping $1$-box avoiding permutations over $[4]$ to binary strings without singletons.}\label{tabMap}
\end{table}

We say that
a word $w = w_1 \ldots w_n \in [\ell]^n$ is $k$-{\em smooth} if $|w_i-w_{i+1}| \leq k$ 
for $1 \leq i < n$. Thus in our notation, 
$w\in [\ell]^n$ is $k$-smooth if $k\mbox{-}\bond[w] =n-1$.  Let 
$M_{n,k,\ell}$ denote the number of $w \in [\ell]^n$ such 
that $k\mbox{-}\bond[w] = n-1$ and 
$sm_{\ell,k}(t) = 1+ \sum_{n \geq 1} M_{n,k,\ell}t^n$.  Clearly,
$$B_{\ell,k}(1/x,xt) = 1 + \sum_{n \geq 1} \sum_{w \in [\ell]^n} 
x^{n -k\mbox{-}\bond[w]}t^n$$
so that 
$$C_{\ell,k}(x,t)= \frac{1}{x}(B_{\ell,k}(1/x,xt)-1)
\sum_{n \geq 1} \sum_{w \in [\ell]^n} 
x^{n -1-k\mbox{-}\bond[w]}t^n.$$
Hence 
\begin{equation*}
sm_{\ell,k}(t) = 1+ C_{\ell,k}(0,t).
\end{equation*}

\begin{table}
\begin{center}
\begin{tabular}{c|c}
$\ell$ & generating function for words $w \in [\ell]^n$ such 
that $2\mbox{-}\bond[w] =n-1$.  \\
\hline
& \\[-3mm]
$sm_{4,2}(t)$ & $\frac{1+t}{1-3t-2t^2}$ \\[2mm]
\hline
& \\[-3mm]
$sm_{5,2}(t)$ & $\frac{1+t-t^2}{1-4t+t^3}$ \\[2mm]
\hline
& \\[-3mm]
$sm_{6,2}(t)$ & $\frac{1+2t-t^2-t^3}{1-4t-t^2+t^3}$\\[2mm]
\hline
& \\[-3mm]
$sm_{7,2}(t)$ & $\frac{1+2t-4t^2-2t^3+2t^4}{1-5t+2t^2+4t^3-2t^4}$
\end{tabular}
\end{center}
\caption{Distribution of words $w \in [\ell]^n$ such 
that $2\mbox{-}\bond[w]=n-1$,  $\ell=4,5,6,7$.}\label{tab1.123}
\end{table}

We have used our generating functions for $B_{\ell,2}(x,t)$ to 
compute $sm_{\ell,2}(t)$ for $\ell=4,5,6,7$, which we record in Table \ref{tab223}. In the case $\ell=4$, our objects match a combinatorial interpretation for the sequence A055099. For the sequence A126392, the generating function 
$sm_{5,2}(t) = \frac{1+t-t^2}{1-4t+t^3}$ was conjectured by 
Colin Barker, so we have proved his conjecture. The sequences 
 A126393 and  A126394 were apparently computed 
from their combinatorial definitions by R. H. Hardin, 
so that we now have found explicit formulas for their generating functions.

\begin{table}
\begin{center}
\begin{tabular}{c|c|c}
$\ell$ & number of words $w \in [\ell]^n$ such 
that $2\mbox{-}\bond[w]=n-1$  & sequence in \cite{oeis} \\
\hline
4 & 1, 4, 14, 50, 178, 634, 2258, 8042, 28642,102010, ... & A055099, $n \geq 0$\\
\hline
5 & 1, 5, 19, 75, 295, 1161, 4569, 17981, 70763, 278483, ... & A126392, $n\geq 0$ \\
\hline
6 & 1, 6, 24, 100, 418, 1748, 7310, 30570, 127842, 534628, ... & A126393, $n \geq 0$\\
\hline
7 & 1, 7, 29, 125, 543, 2363, 10287, 44787, 194995, 848979, ... & A126394, $n \geq 0$\\
\end{tabular}
\end{center}
\caption{Number of words $w \in [\ell]^n$ such that $k\mbox{-}\bond[w]=n-1$ for  $n$ up to 9.}\label{tab223}
\end{table}

One can also modify the proof of Theorem \ref{thm:1boxgf} to 
find the generating function for the distribution of 
$(1,k)\mbox{-rec}(w)$ for $w \in [\ell]^*$. That is, 
suppose $k \geq 2$, and for $1 \leq i,j \leq \ell$,  
$$B_{\ell,k}^{(ij)} = \sum_{w \in ij[\ell]^*} WT_k(w)$$ 
where $WT_k(w) = x^{(1,k)\mbox{-rec}(w)}t^{|w|}$.
 Then we claim that 
for all $1 \leq i,j \leq \ell$, 
\begin{eqnarray}\label{kboxwdrec}
&&B_{\ell,k}^{(ij)}(x,t) = x^{2\chi(|i-j]\leq k)}t^2 + \\
&&\sum_{k=1}^{\ell} 
(t\chi(|i-j| > k) + xt\chi(|i-j| \leq k)\chi(|j-k| \leq k) + \nonumber \\
&& \ \ \ \ \ \ \ \ \ \ \ x^2t
\chi(|i-j| \leq k)\chi(|j-k| > k))B^{(jk)}_{\ell,k}(x,t). \nonumber 
\end{eqnarray}
That is, the words in $ij[\ell]^*$ are of the 
form $ij$ plus words $ijmv$ where $m \in [\ell]$ and 
$v \in [\ell]^*$. Now 
\begin{equation*}
WT_k[ij] = \begin{cases} t^2 & \mbox{ if } |i-j|>k  \mbox{ and} \\ 
x^2t^2 & \mbox{ if } |i-j|\leq k.
\end{cases}
\end{equation*}
Similarly, 
\begin{equation*}
WT_k[ijkv] = \begin{cases} tWT_k[jkv] & \mbox{ if } |i-j|>k,  \\ 
xt WT_k[jkv] & \mbox{ if } |i-j|\leq k \mbox{ and }  |j-k|\leq   k,  \mbox{ and}\\x^2t WT_k[jkv] & \mbox{ if } |i-j|\leq k \mbox{ and }  |j-k|>  k.
\end{cases}
\end{equation*}

The set of equations of the form (\ref{kboxwdrec}) 
can be written out in matrix form. That is 
let 
$\vec{B}_{\ell,1}$ be the row vector of length $\ell^2$ of 
the $B^{(ij)}_{\ell,1}(t,x)$ where the elements are listed in 
the lexicographic order of the pairs $(ij)$. 
Let 
$\vec{I}_{\ell,k}$ be the row vector of length $\ell^2$ of the 
terms $t^2x^{2\chi(|i-j| \leq k)}$ again listed in 
the lexicographic order on the pairs $ij$. For example, 
$$ \vec{I}_{4,2}=(x^2t^2,x^2t^2,x^2t^2,t^2,x^2t^2,x^2t^2,x^2t^2,x^2t^2,
x^2t^2,x^2t^2,x^2t^2,x^2t^2,t^2,x^2t^2,x^2t^2,x^2t^2).$$
Then one can write a set of equations of the form (\ref{kboxwdrec}) in 
the form 
$$(\vec{I}_{\ell,x})^T = \mathbb{B}_{\ell,k} (\vec{B}_{\ell,k})^T$$
where 
$ \mathbb{B}_{\ell,k}$ is an $\ell^2\times \ell^2$ matrix. 
For example, $\mathbb{B}_{4,2}$ is the matrix

{\tiny 
$$\left(
\begin{array}{cccccccccccccccc}
xt-1 & xt & xt & x^2 t & 0 & 0 & 0 & 0 & 0 & 0 & 0 & 0 & 0 & 0 & 0 & 0  \\
0  & -1 & 0  & 0  & xt & xt  & xt & xt & 0 & 0 & 0 & 0 & 0 & 0 & 0 & 0  \\
0 & 0  & -1 & 0  & 0 & 0 & 0 & 0  & xt & xt & xt & xt & 0 & 0 & 0 & 0 \\
0 & 0 & 0 & -1 & 0 & 0 & 0 & 0 & 0 & 0 & 0 & 0 & t & t & t & t \\
xt & xt &xt & x^2t & -1 & 0 & 0 & 0 & 0 & 0 & 0 & 0 & 0 & 0 & 0 & 0  \\
0  & 0 & 0  & 0 & xt & xt-1  & xt & xt &  0 & 0 & 0 & 0 & 0 & 0 & 0 & 0  \\
0  & 0 & 0  & 0 & 0 & 0 & -1 & 0  & xt & xt & xt & xt & 0 & 0 & 0 & 0 \\
0  & 0 & 0  & 0 & 0 & 0 & 0 & -1 & 0 & 0 & 0 & 0 & x^2t & xt & xt & xt \\
xt & xt & xt & x^2t & 0 & 0 & 0 & 0 & -1 & 0 & 0 & 0 & 0 & 0 & 0 & 0  \\
0  & 0 & 0  & 0 & xt & xt   & xt & xt &  0 & -1 & 0 & 0 & 0 & 0 & 0 & 0  \\
0  & 0 & 0  & 0  & 0 & 0 & 0 & 0  & xt & xt & xt -1& xt & 0 & 0 & 0 & 0 \\
0  & 0 & 0 & 0 & 0 & 0 & 0 & 0 & 0 & 0 & 0 & -1 & x^2t & xt & xt & xt \\
t & t & t & t & 0 & 0 & 0 & 0 & 0 & 0 & 0 & 0 & -1 & 0 & 0 & 0  \\
0  & 0 & 0  & 0 & xt & xt  & xt & xt &  0 & 0 & 0 & 0 & 0 & -1 & 0 & 0  \\
0  & 0 & 0  & 0  & 0 & 0 & 0 & 0  & xt & xt & xt & xt & 0 & 0 & -1 & 0 \\
0 & 0 & 0 & 0 & 0 & 0 & 0 & 0 & 0 & 0 & 0 & 0 & x^2t & xt & xt & xt-1 
\end{array}\right)$$}

Note that $ \mathbb{B}_{\ell,k}$ is invertible since 
setting $x=t=0$ in $ \mathbb{B}_{\ell,k}$ will give 
the $\ell \times \ell$ diagonal matrix with $-1$s on 
the diagonal. Thus 
$$(\vec{B}_{\ell,k})^T = \mathbb{B}_{\ell,k}^{-1} (\vec{I}_{\ell,k})^T.$$
Let $\vec{1}_{\ell,1}$ denote the vector of length $\ell^2$ consisting 
of all 1s. Then 
$$\sum_{1 \leq i,j \leq \ell} B^{(ij)}_{\ell,1}(x,t) = 
\vec{1}_{\ell,1} \mathbb{B}_{\ell,k}^{-1} (\vec{I}_{\ell,k})^T.$$
Taking into account the empty word and all the 
words of length 1 will yeild the following theorem.

\begin{theorem}\label{thm:kboxgf} For all $\ell \geq 2$, 
\begin{equation*}
B_{\ell,k}(x,t) = 1 + \ell t + \vec{1}_{\ell,k} 
\mathbb{B}_{\ell,k}^{-1} (\vec{I}_{\ell,k})^T.
\end{equation*}
\end{theorem}

Note that $B_{\ell,2}(x,t)=\frac{1}{1-\ell xt}$ for $\ell = 1,2,3$ since in 
such words every letter matches $(1,2)$-rectangle pattern. 
We have used Theorem \ref{thm:kboxgf} to compute 
$B_{\ell,2}(x,t)$ for $\ell = 4,5$:

\begin{equation*}\label{eq:B42}
B_{4,2}(x,t)= \frac{1-3 t (-1+x)+6 t^3 (-1+x)^2 x+4 t^4 (-1+x)^2 x^2+t^2 \left(2+9 x-11 x^2\right)}{1-t-3 t x-3 t^2 (-1+x) x+2 t^3 (-1+x)
x^2};
\end{equation*}

\begin{eqnarray*}\label{eq:B52}
&&B_{5,2}(x,t) = \\
&&-\frac{1+t (-1+x) \left(-4+t \left(16 x+t (-1+x) \left(-1+x \left(-2+t^3 (-1+x) x^2+4 t (1+x)\right)\right)\right)\right)}{-1+t
\left(1+4 x+t (-1+x) \left(-1+x \left(3+t^3 (-1+x) x^2+t (4+x)\right)\right)\right)}. \nonumber
\end{eqnarray*}
Using the generating functions above, we have computed 
some of the initial terms in their Taylor series expansions:

\begin{tiny} 
\begin{eqnarray*}
&&B_{4,2}(x,t)  =1+4 t+2 \left(1+7 x^2\right) t^2+2 \left(1+6 x^2+25 x^3\right) t^3+2 \left(1+7 x^2+22 x^3+98 x^4\right) t^4+\\
&&2 \left(1+8 x^2+27 x^3+93
x^4+383 x^5\right) t^5+2 \left(1+9 x^2+32 x^3+117 x^4+396 x^5+1493 x^6\right) t^6+\\
&&2 \left(1+10 x^2+37 x^3+142 x^4+519 x^5+1659 x^6+5824 x^7\right)
t^7+\\
&&2 \left(1+11 x^2+42 x^3+168 x^4+652 x^5+2247 x^6+6930 x^7+22717 x^8\right) t^8+\cdots;
\end{eqnarray*}

\begin{eqnarray*}
&&B_{5,2}(x,t) = 1+5 t+\left(6+19 x^2\right) t^2+5 \left(2+8 x^2+15 x^3\right) t^3+\left(16+88 x^2+160 x^3+361 x^4\right) t^4+\\
&&\left(26+176 x^2+358 x^3+876
x^4+1689 x^5\right) t^5+\\
&&\left(42+342 x^2+724 x^3+2106 x^4+4496 x^5+7915 x^6\right) t^6+\\
&&\left(68+644 x^2+1416 x^3+4586 x^4+11328 x^5+22976 x^6+37107
x^7\right) t^7+\\
&&\left(110+1190 x^2+2680 x^3+9562 x^4+25712 x^5+60762 x^6+116672 x^7+173937 x^8\right) t^8+\cdots.
\end{eqnarray*}
\end{tiny}
 
We also can compute the generating functions of the number of 
words that avoid the $(2,1)$-rectangle patterns for words 
$w \in [5]^*$. That is, we have that 
\begin{eqnarray*}
B_{5,2}(0,t) &=& \frac{1+4 t-t^3}{1-t-t^2} \\
&=& 1+5 t+6 t^2+10 t^3+16 t^4+26 t^5+42 t^6+68 t^7+110 t^8+\cdots.
\end{eqnarray*}

Note that 
$$\overline{B}_{\ell,k}(x,t) := B_{\ell,k}(1/x,xt) = 
\sum_{w \in [\ell]^*} x^{n-((1,k)\mbox{-rec}(w))}t^n$$ 
so that $\overline{B}_{\ell,k}(0,t)$ is the generating function 
of all words in $w =w_1 \ldots w_n \in[k]^*$ such that 
 $(1,k)\mbox{-rec}(w)=n$, i.e. each letter of $w$ differs from at least one neighbor by $k$ or less. We have computed $\overline{B}_{\ell,2}(0,t)$ for 
$k = 4,5$. 

\begin{equation*}
\overline{B}_{4,2}(0,t) = \frac{1-3 t+11 t^2+6 t^3+4 t^4}{1-3 t-3 t^2-2 t^3}.
\end{equation*}
The initial terms of this series are 
$1,0,14,50,196,766,2986,11648,44343,177218,691252,\ldots$. This sequence 
does not appear in the OEIS.

\begin{equation*}
\overline{B}_{5,1}(0,t) = \frac{1-3t+9t^2-4t^3+6t^4+4t^6}{1-3t-4t^2-6t^4-4t^5-4t^6}.
\end{equation*}
The initial terms of this series are 
$1,0,19,75,361,1689,7915,37107,173937,815345,\ldots$. This sequence 
also does not appear in the OEIS. 

Our methods obviously extend to allow us to write a matrix equation 
for the generating function 
$B_{\ell,a,b}(x,t) = \sum_{w \in [\ell]^*} x^{(a,b)\mbox{-rec}(w)} t^{|w|}.$
However, it becomes computationally unfeasiable even in the case 
of $2\mbox{-box}(w)$. That is, one has to keep track of the first four 
letters to be able to compute the necessary recursions. 
For example, let 
$$B^{rstu}_{\ell,2\mbox{-box}}(x,t) =  
\sum_{w \in rstu [\ell]^*} x^{2\mbox{-box}(w)} t^{|w|},$$ where $rstu [\ell]^*$ is the set of all words over $[\ell]$ that begin with letters $rstu$.
Then it is easy to 
see that 

$$B^{rstu}_{\ell,2\mbox{-box}}(x,t) = x^{2\mbox{-box}(rstu)}t^4 + \sum_{v =1}^{\ell} 
\theta(rstuv) B^{stuv}_{\ell,2\mbox{-box}}(x,t)$$
where $\theta(rstuv)$ is computed according the following 
four cases. \\
\ \\
{\bf Case 1.} $|r-s| > 2$ and $|r-t| > 2$. In this case, $\theta(rstuv) =t$.\\
\ \\
{\bf Case 2.} $|r-s| > 2$ and $|r-t| \leq 2$. In this case, $\theta(rstuv) =xt$ if $t$ matches the 2-box pattern in 
$stuv$ and $\theta(rstuv) =x^2t$ if $t$ does not match the 2-box pattern in 
$stuv$. That is, for any word $w \in [\ell]^*$, 
the presence of $r$ does not effect whether $s$ will 
match the 2-box pattern in $rstuw$, but it does effect the question of 
whether $t$ matches the 2-box pattern in $rstuvw$.\\
\ \\
{\bf Case 3.} $|r-s| \leq 2$ and $|r-t| > 2$. In this case, $\theta(rstuv) =xt$ if $s$ matches the 2-box pattern in $stu$ and $\theta(rstuv) =x^2t$ if $s$ does not match the 2-box pattern in 
$stu$. That is, for any word $w \in [\ell]^*$, 
the presence of $r$ does not effect whether $t$ will 
match the 2-box pattern in $rstuw$, but it does effect the question of 
whether $s$ matches the 2-box pattern in $rstuvw$.\\
\ \\
{\bf Case 4.} $|r-s| \leq 2$ and $|r-t| \leq 2$. In this case $\theta(rstuv) =xt$ if 
both $s$ and $t$ match the 2-box pattern in $stuv$, $\theta(rstuv) =x^2t$ if 
exactly one of $s$ and $t$ match the 2-box pattern in $stuv$, and 
$\theta(rstuv) =x^3t$ if 
neither $s$ nor $t$ match the 2-box pattern in $stuv$.\\
\ \\
This recursion allows us to write a simple matrix type equation for 
the generating function $B_{\ell,2\mbox{-box}}(x,t)$; however, 
it requires that we have to invert an $\ell^4 \times \ell^4$ matrix 
which is not really feasible even for small $\ell$. Indeed, the 
generating function $B_{\ell,2\mbox{-box}}(x,t)$ is trivial for $\ell \leq 3$, 
so the smallest non-trivial $\ell$ is $\ell =4$ which requires we would have 
to invert a $4^4 \times 4^4$-matrix.

\section{Conclusion}

The goal of this paper was to introduce $k$-box patterns and to study them, mainly in the case of $k=1$, on permutations and words. In the upcoming paper  \cite{kitrem3}, we study 1-box patterns on pattern-avoiding permutations (more precisely, on $132$-avoiding permutations and on {\em separable permutations}).

\end{document}